\documentclass[11pt]{amsart}
\usepackage{amssymb}
\usepackage{amsmath}
\usepackage{comment}
\usepackage{latexsym}
\usepackage{amsfonts}
\usepackage{graphicx}
\usepackage[normalem]{ulem}
\usepackage[utf8]{inputenc}
\usepackage[bookmarks=true,pdfstartview=FitH, pdfborder={0 0 0},
colorlinks=true,citecolor=blue, linkcolor=blue, urlcolor=blue, hypertexnames=false]{hyperref}
\newtheorem{theorem}{Theorem}
\newtheorem{acknowledgement}[theorem]{Acknowledgement}

\newtheorem{corollary}[theorem]{Corollary}
\newtheorem{definition}[theorem]{Definition}

\newtheorem{lemma}[theorem]{Lemma}

\newtheorem{proposition}[theorem]{Proposition}
\newtheorem{question}[theorem]{Question}
\newtheorem{remark}[theorem]{Remark}

\newcommand{\abs}[1]{\left\lvert #1 \right\rvert}

\newcommand{\F}{\mathbb{F}}
\newcommand{\Q}{\mathbb{Q}}
\newcommand{\R}{\mathbb{R}}
\newcommand{\RP}{\mathbb{R}P}

\newcommand{\tld}[1]{\widetilde{#1}}
\newcommand{\Z}{\mathbb{Z}}

\protected\def\vts{
  \ifmmode
    \mskip0.5\thinmuskip
  \else
    \ifhmode
      \kern0.08334em
    \fi
  \fi
}

\begin{document}

\title[Normal-Euler excess and packing]{Normal-Euler excess and codimension-zero packing in $4$-manifolds}

\author[B. Chow]{Bennett Chow}
\address{Department of Mathematics, University of California, San Diego, California 92093, USA.}

\author[M. Freedman]{Michael Freedman}
\address{Department of Mathematics, Harvard University, CMSA, Cambridge, Massachusetts 02139, USA.}

\date{\today}

\begin{abstract}
Let \(M\) be a closed connected oriented topological \(4\)-manifold. We prove
that if \(F_1,\dots,F_r\subset M\) are pairwise disjoint connected closed locally flat
nonorientable surfaces with nonorientable genera \(g_i\), same-sign twisted normal
Euler numbers \(e_i\), and
\[
[F_1]+\cdots+[F_r]=0\in H_2(M;\F_2),
\]
then the normal-Euler excess
\[
\sum_{i=1}^r\bigl(\abs{e_i}-2g_i\bigr)
\]
is bounded above by a constant depending only on \(M\). Thus same-sign
mod-\(2\)-null families of disjoint nonorientable surfaces in a fixed ambient
\(4\)-manifold have uniformly bounded total excess over Massey's \(S^4\) bound.
The proof combines tubing with the signature and Euler-characteristic formulas
for \(2\)-fold branched covers.

As corollaries, every closed oriented topological \(4\)-manifold contains only
finitely many pairwise disjoint locally flat copies of \(\RP^2\) with
\(\abs{e}>2\), and only finitely many pairwise disjoint tubular neighborhoods
modeled on real \(2\)-plane bundles over \(\RP^2\) whose total spaces are
orientable and whose twisted Euler numbers have absolute value greater than
\(2\). When \(M\) is a homology \(4\)-sphere, the ambient error term vanishes and
the theorem recovers Massey's sharp inequality \(\abs{e(F)}\le 2g(F)\) for
nonorientable surfaces in \(S^4\).

We also give a connected-boundary counterexample to a natural codimension-zero
packing question. Let \(P=\RP^3\setminus\operatorname{int}B^3\) and let
\(W=P\,\widetilde{\times}\,[0,1]\) denote the rounded product. Then \(W\) is a
compact connected oriented \(4\)-manifold with connected boundary. Whenever a
compact oriented \(4\)-manifold contains a bicollared copy of \(\RP^3\), it
contains arbitrarily many pairwise disjoint codimension-zero copies of \(W\);
in particular this holds in \(\RP^3\times S^1\). Nevertheless \(W\) does not
embed in \(\R^4\).
\end{abstract}

\maketitle

\section{Introduction}

A classical theorem of Massey \cite{Massey}, completing Whitney's conjecture by
an application of the Atiyah--Singer index theorem, states that if
\(F\subset S^4\) is a connected smoothly embedded nonorientable surface of
nonorientable genus \(g\), and if \(e(F)\) denotes the twisted normal Euler number
of \(F\), then
\[
\abs{e(F)}\le 2g.
\]
This inequality is sharp: for \(g=1\), equality is realized by the Veronese
surface, and in fact Massey proved more generally that for each nonorientable
genus \(g\), every value
\[
e(F)\in \{-2g,-2g+4,\ldots,2g\}
\]
occurs for some smooth embedding \(F\subset S^4\). It is natural to ask what
survives when the ambient manifold is an arbitrary closed oriented topological
\(4\)-manifold and the surface is locally flat topologically embedded.

For a fixed closed oriented topological \(4\)-manifold \(M\), however, one should
not expect a verbatim analogue of the \(S^4\) bound for an individual surface.
The purpose of this paper is to show that a strong substitute does hold for
disjoint families. If
\[
F_1,\dots,F_r\subset M
\]
are pairwise disjoint connected closed locally flat topologically embedded
nonorientable surfaces whose mod-\(2\) homology classes sum to zero and whose
twisted normal Euler numbers all have the same sign, then the total excess of
the normal Euler numbers over the Massey term \(2g_i\) is bounded by a constant
depending only on \(M\). In particular, a fixed closed oriented \(4\)-manifold
contains only finitely many disjoint projective planes with \(\abs{e}>2\).

One of the original motivations for the projective-plane case came from questions
about disjoint plane-bundle neighborhoods in \(4\)-manifolds arising in geometric
applications. A particularly relevant example comes from singularity formation
in compact \(4\)-dimensional Ricci flow. A singularity model is, by definition, a
complete Ricci flow obtained as a limit of rescalings of a solution, and a priori
it need not have bounded curvature. In \cite{CFSZ,CFSZCorr}, it was shown that if
a \(4\)-dimensional singularity model is also a steady Ricci soliton, then it must
in fact have bounded curvature. The proof is by contradiction: if the curvature
were unbounded, then a further pointed limit would be a nonflat Ricci-flat ALE
\(4\)-manifold; by taking a fixed compact core in this ALE limit and using the
pointed convergence maps at sufficiently separated basepoints, one obtains
arbitrarily many pairwise disjoint topological copies of that core in the
singularity model, and hence in the underlying compact \(4\)-manifold of the
original Ricci flow. This is ruled out by homological and index-theoretic
arguments. Thus, in that setting, a topological argument yields an analytic
estimate for \(4\)-dimensional Ricci flow.

However, bounded curvature still does not exclude the formation of blips in a
steady soliton singularity model; see \cite[Open Problem 5.1]{CKM}. The results
of the present paper rule out, for topological reasons, those hypothetical blips
whose compact cores contain, or are modeled on, real \(2\)-plane bundles over
\(\RP^2\) with orientable total space and twisted Euler number of absolute value
greater than \(2\), whenever the geometric limiting argument would force
arbitrarily many disjoint copies of such a core in the original compact
\(4\)-manifold. The results proved here are purely topological.

To state the theorem precisely, we fix conventions for twisted normal Euler
numbers.

\begin{definition}
Let \(X\) be a closed connected nonorientable surface locally flat topologically
embedded in an oriented topological \(4\)-manifold \(M\). Such an embedding has a
topological normal \(2\)-plane bundle; see, for example,
\cite[Theorem~9.6]{FreedmanQuinn}. Let
\[
\nu_M(X)\to X
\]
denote this normal bundle. Let
\[
\omega:\pi_1(X)\to\{\pm1\}
\]
be the orientation character of \(X\) (i.e., the homomorphism corresponding to
\(w_1(X)\) under the identification
\(H^1(X;\mathbb{Z}/2)\cong \operatorname{Hom}(\pi_1(X),\mathbb{Z}/2)\)), and let
\(\widetilde{\Z}\) denote the associated local coefficient system on \(X\).\footnote{For
background on the above identification and local coefficient systems, see, for
example, Hatcher's book \cite[Sections~3.1 and~3.H]{Hatcher}.}

Let \(\mathcal O_X\) denote the orientation local system of \(X\), and let
\(\mathcal O_{\nu}\) denote the orientation local system of \(\nu_M(X)\). Since
\(M\) is oriented and \(X\subset M\) is locally flat of codimension two, the
chosen orientation of \(M\) gives a canonical trivialization
\[
\mathcal O_X\otimes \mathcal O_{\nu}\cong \mathbb Z.
\]
Equivalently,
\[
w_1(\nu_M(X))=w_1(X),
\]
and the chosen orientation of \(M\) fixes the identifications of
\(\widetilde{\Z}\) with both \(\mathcal O_X\) and \(\mathcal O_{\nu}\).

Let \(p:D(\nu_M(X))\to X\) be the disk bundle, let \(S(\nu_M(X))\subset
D(\nu_M(X))\) be the sphere bundle, let
\[
u_{\nu_M(X)}\in H^2\!\bigl(D(\nu_M(X)),S(\nu_M(X));p^*\widetilde{\Z}\bigr)
\]
be the Thom class, and let \(s:X\to D(\nu_M(X))\) be the zero section.\footnote{For
background on Thom classes and the Thom isomorphism, see, for example,
\cite[Section~4.D]{Hatcher}.}
We define the twisted Euler class of \(\nu_M(X)\) by
\[
\chi(\nu_M(X)):=s^*u_{\nu_M(X)}\in H^2(X;\widetilde{\Z}).
\]

Since \(X\) is a closed connected \(2\)-manifold, it has a canonical twisted
fundamental class
\[
[X]_{\mathrm{tw}}\in H_2(X;\widetilde{\Z}),
\]
namely the Poincar\'e dual of \(1\in H^0(X;\Z)\).\footnote{For background on
Poincar\'e duality with orientation local systems, see, for example,
\cite[Sections~3.3 and~3.H]{Hatcher}.}
We define the twisted normal Euler number of \(X\subset M\) by
\[
e(X):=\langle \chi(\nu_M(X)),[X]_{\mathrm{tw}}\rangle\in \Z.
\]
Equivalently, after choosing any smooth structure on the abstract surface \(X\)
and any compatible smooth structure on the bundle
\[
\nu_M(X)\to X,
\]
\(e(X)\) is the algebraic zero count of a generic section of \(\nu_M(X)\), with
signs taken using the local coefficient system \(\widetilde{\Z}\).
\end{definition}

\begin{remark}
The chosen orientation of \(M\) fixes the sign convention for \(e(X)\); reversing
the orientation of \(M\) reverses the sign of \(e(X)\). In particular, for a fixed
oriented ambient manifold, the same-sign hypothesis in Theorem~\ref{thm:main} is
meaningful.
\end{remark}

Our main result is the following ambient-manifold version of Massey's inequality.
Here and below, we say that integers have the same sign if they are all
nonnegative or all nonpositive.

\begin{theorem}\label{thm:main}
Let \(M\) be a closed connected oriented topological \(4\)-manifold. Let
\[
F_1,\dots,F_r \subset M
\]
be pairwise disjoint connected closed locally flat topologically embedded
nonorientable surfaces. For each \(i\), let \(g_i\) be the nonorientable genus of
\(F_i\), so that \(\chi(F_i)=2-g_i\), and let
\[
e_i:=e(F_i)\in \Z
\]
be the twisted normal Euler numbers.

Assume that the integers \(e_1,\dots,e_r\) all have the same sign and that
\[
[F_1]+\cdots+[F_r]=0 \in H_2(M;\F_2).
\]
Then
\begin{equation}\label{eq:mainbound}
\sum_{i=1}^r \abs{e_i}
\le
2\sum_{i=1}^r g_i
+4\abs{\sigma(M)}+8b_1(M;\F_2)+4\chi(M)-8,
\end{equation}
where \(\sigma(M)\) denotes the signature of \(M\). Equivalently,
\begin{equation}\label{eq:excessbound}
\sum_{i=1}^r \bigl(\abs{e_i}-2g_i\bigr)
\le
4\abs{\sigma(M)}+8b_1(M;\F_2)+4\chi(M)-8.
\end{equation}
In particular, the right-hand side depends only on \(M\), so the bound is
uniform over all such embedded surfaces.
\end{theorem}

The two hypotheses serve different purposes. After tubing the \(F_i\) together to
a connected surface \(F\), the condition
\[
[F_1]+\cdots+[F_r]=0\in H_2(M;\F_2)
\]
is exactly what allows one to form a branched double cover over \(F\). The
same-sign assumption prevents cancellation in the identity
\[
e(F)=e(F_1)+\cdots+e(F_r),
\]
so that
\[
\abs{e(F)}=\sum_{i=1}^r \abs{e_i}.
\]

We will refer to the quantity on the left-hand side of \eqref{eq:excessbound} as
the \emph{normal-Euler excess} of the family \(F_1,\dots,F_r\).

When \(M=S^4\) (or, more generally, when \(M\) is a homology \(4\)-sphere), the
ambient error term on the right-hand side vanishes. In that case
Theorem~\ref{thm:main} recovers Massey's theorem \cite{Massey}: for any smoothly
embedded connected nonorientable surface \(F\subset S^4\) of nonorientable genus
\(g\),
\[
\abs{e(F)}\le 2g.
\]
Thus Theorem~\ref{thm:main} may be viewed as a closed-\(4\)-manifold
generalization of Massey's inequality, with an error term determined entirely by
the ambient manifold.

Two immediate consequences are the following.

\begin{corollary}\label{cor:projective-planes}
Let \(M\) be a closed connected oriented topological \(4\)-manifold. Then there
exists a constant \(B(M)\) such that \(M\) contains at most \(B(M)\) pairwise
disjoint locally flat topologically embedded copies of \(\RP^2\) whose twisted
normal Euler numbers satisfy \(\abs{e}>2\).
\end{corollary}

If \(E\to \RP^2\) is a real \(2\)-plane bundle whose total space is orientable
(equivalently, \(w_1(E)=w_1(T\RP^2)\)), let
\[
\chi(E)\in H^2(\RP^2;\widetilde{\Z})
\]
denote its twisted Euler class, and write
\[
e_{\mathrm{tw}}(E):=\langle \chi(E),[\RP^2]_{\mathrm{tw}}\rangle\in \Z
\]
for its twisted Euler number.

\begin{corollary}\label{cor:plane-bundles}
Let \(M\) be a closed connected oriented topological \(4\)-manifold. Then \(M\)
contains only finitely many pairwise disjoint closed tubular neighborhoods of
locally flat topologically embedded copies of \(\RP^2\) with twisted normal Euler
number of absolute value greater than \(2\). In particular, \(M\) contains only
finitely many pairwise disjoint open subsets each homeomorphic to the total space
of a real \(2\)-plane bundle over \(\RP^2\) whose total space is orientable and
whose twisted Euler number has absolute value greater than \(2\).
\end{corollary}

The theorem is therefore not merely a counting statement for projective planes; it
gives a uniform excess bound for arbitrary disjoint nonorientable surfaces.

The proof is short once the relevant branched-cover input is isolated. One first
tubes the surfaces \(F_i\) together to obtain a connected closed nonorientable
surface \(F\) with
\[
[F]=[F_1]+\cdots+[F_r]\in H_2(M;\F_2),\qquad
e(F)=\sum_{i=1}^r e_i,\qquad
g(F)=\sum_{i=1}^r g_i.
\]
The hypothesis \([F]=0\) then produces a connected \(2\)-fold branched cover
\[
p\colon N\to M
\]
where \(N\) is a connected closed \(4\)-manifold, branched over \(F\). One can
picture this cover on the exterior of \(F\). If \(U\) is a tubular neighborhood of
\(F\) and \(E:=M\setminus \operatorname{int}(U)\), then the hypothesis
\([F]=0\in H_2(M;\F_2)\) produces a class in \(H^1(E;\F_2)\) that evaluates
nontrivially on each meridian circle of \(F\). This class determines a connected
double cover of \(E\), and Fox's completion construction fills the cover back in
across \(U\) to give the desired branched cover over \(M\).

The key identities are
\[
\abs{\sigma(N)-2\sigma(M)}=\tfrac12\abs{e(F)}
\qquad\text{and}\qquad
\chi(N)=2\chi(M)-\chi(F),
\]
together with the estimate
\[
b_1(N;\F_2)\le 2b_1(M;\F_2).
\]
These bounds control \(b_2(N)\) in terms of \(g(F)\), and hence yield the desired
estimate for \(e(F)\).

The finiteness conclusions above concern disk-bundle neighborhoods over
nonorientable surfaces with positive normal-Euler excess. They do not, by
themselves, imply that an arbitrary compact oriented \(4\)-manifold that packs
asymptotically into a compact ambient \(4\)-manifold must embed in \(\R^4\). In the
last section we make this distinction explicit by constructing a compact connected
oriented \(4\)-manifold \(W\) with connected boundary that packs arbitrarily many
times into a compact oriented \(4\)-manifold but does not embed in \(\R^4\). The
example is
\[
W=(\RP^3\setminus \operatorname{int}B^3)\,\widetilde{\times}\,[0,1],
\]
the rounded product of a punctured \(\RP^3\) with an interval. The nonembedding
obstruction comes from Hantzsche's boundary linking-form obstruction. In the
smooth category, one can also see the nonembedding from Whitney's congruence for
the twisted normal Euler number of \(\RP^2\subset \R^4\).

The paper is organized as follows. In Section~2 we record the linear-algebra
lemma, the tubing construction, the branched-cover formulas, and the Betti-number
estimate needed in the proof. Section~3 proves Theorem~\ref{thm:main},
Section~4 deduces the projective-plane and plane-bundle corollaries, and
Section~5 answers a natural connected-boundary codimension-zero packing question
in the negative.

\section{Preliminaries}

We collect the three ingredients used in the proof.

\subsection{A linear-algebra lemma}

\begin{lemma}\label{lem:zero-sum}
Let \(V\) be a vector space over \(\F_2\) of dimension \(k\), e.g., \((\Z_2)^k\).
Given any \(m\) elements
\[
x_1,\dots,x_m\in V,
\]
there exists a (possibly empty) subcollection whose sum is zero and whose
cardinality is at least \(m-k\). In particular, if \(m>k\), then there exists a
\emph{nonempty} zero-sum subcollection of size at least \(m-k\).
\end{lemma}

\begin{proof}
Let
\[
r:=\dim \operatorname{span}\{x_1,\dots,x_m\}\le k,
\]
and choose indices \(i_1,\dots,i_r\) such that
\[
x_{i_1},\dots,x_{i_r}
\]
form a basis of \(\operatorname{span}\{x_1,\dots,x_m\}\).

Let
\[
J:=\{1,\dots,m\}\setminus \{i_1,\dots,i_r\}.
\]
Then \(|J|=m-r\). Since each \(x_j\) with \(j\in J\) lies in
\(\operatorname{span}\{x_{i_1},\dots,x_{i_r}\}\), the vector
\[
y:=\sum_{j\in J} x_j
\]
also lies in that span. Because we are working over \(\F_2\), the coordinates of
\(y\) in the basis \(x_{i_1},\dots,x_{i_r}\) are either \(0\) or \(1\). Hence there
exists a subset \(I\subseteq \{i_1,\dots,i_r\}\) such that
\[
y=\sum_{i\in I} x_i.
\]
Therefore
\[
\sum_{j\in J} x_j+\sum_{i\in I} x_i=0.
\]
Thus the subcollection indexed by \(I\cup J\) has zero sum. Since
\(I\cap J=\varnothing\), its cardinality is
\[
|I\cup J|=|I|+|J|\ge |J|=m-r\ge m-k.
\]

If \(m>k\), then \(m-k>0\), so this zero-sum subcollection is nonempty.
\end{proof}

\subsection{Tubing disjoint surfaces}

\begin{lemma}\label{lem:tubing}
Let \(F_1,\dots,F_r\subset M\) be pairwise disjoint connected closed locally flat
topologically embedded nonorientable surfaces in a connected oriented topological
\(4\)-manifold \(M\). Then, after choosing disjoint embedded arcs joining the
\(F_i\), one can form a connected closed locally flat topologically embedded
ambient connected sum
\[
F := F_1\#\cdots\# F_r \subset M
\]
with the following properties:
\begin{align*}
[F] &= [F_1]+\cdots+[F_r] \in H_2(M;\F_2),\\
\chi(F) &= \chi(F_1)+\cdots+\chi(F_r)-2(r-1),\\
g(F) &= g_1+\cdots+g_r,\\
e(F) &= e(F_1)+\cdots+e(F_r).
\end{align*}
\end{lemma}

\begin{proof}
Choose \(r-1\) pairwise disjoint embedded arcs whose union is a tree connecting
\(F_1,\dots,F_r\). It is enough to describe one tubing step, since iterating that
construction along these arcs gives the general case.

So first suppose \(r=2\). Let \(\gamma\) be an embedded arc joining \(F_1\) to
\(F_2\), disjoint from \(F_1\cup F_2\) except at its endpoints. Choose a closed
tubular neighborhood
\[
U\cong D^3\times [0,1]
\]
of \(\gamma\) such that, for some fixed equatorial disk \(D^2\subset D^3\),
\[
U\cap F_1 = D^2\times \{0\},
\qquad
U\cap F_2 = D^2\times \{1\}.
\]
Define \(F\) by removing the interiors of these two disks and inserting the annulus
\[
\partial D^2\times [0,1]\subset \partial U.
\]
This is the ambient connected sum \(F_1\# F_2\subset M\). The local product model
also shows that \(F\) is locally flat: outside \(U\) this is inherited from
\(F_1\cup F_2\), while inside \(U\cong D^3\times[0,1]\) the inserted annulus
\(\partial D^2\times[0,1]\) is locally flat. Geometrically, one removes one disk
from each surface and joins the resulting boundary circles by a tube. This makes
the formulas for \(\chi(F)\) and \(g(F)\) transparent: two disks are removed, an
annulus is inserted, and the two components become one.

The product \(D^2\times [0,1]\subset U\) is a compact \(3\)-manifold whose boundary
mod \(2\) is exactly the union of the two deleted disks and the inserted annulus.
Hence \(F\) is cobordant mod \(2\) in \(M\) to \(F_1\sqcup F_2\), and therefore
\[
[F]=[F_1]+[F_2]\in H_2(M;\F_2).
\]

Since the tubing removes two disks and adds an annulus, it changes Euler
characteristic by
\[
(-1)+(-1)+0=-2.
\]
Thus
\[
\chi(F)=\chi(F_1)+\chi(F_2)-2.
\]
Because every connected nonorientable surface \(G\) satisfies \(\chi(G)=2-g(G)\),
it follows that
\[
g(F)=g_1+g_2.
\]

It remains to prove additivity of the twisted normal Euler number. Use the
standard zero-count interpretation of the twisted Euler class. Choose generic
sections
\[
s_i\colon F_i\to \nu_M(F_i), \qquad i=1,2,
\]
with isolated zeros, all lying away from the disks \(U\cap F_i\). After a homotopy
supported near those disks, we may assume that each \(s_i\) is nowhere zero there
and, in the product coordinates on \(U\), is equal to the same constant normal
vector field along \(U\cap F_i\). The normal bundle of the annulus
\(\partial D^2\times [0,1]\) inside \(U\) is trivial, so these boundary values
extend across the annulus without introducing any zeros. Thus \(s_1\) and \(s_2\)
glue to a generic section \(s\) of \(\nu_M(F)\) whose zero set is exactly the
disjoint union of the zero sets of \(s_1\) and \(s_2\). Hence the algebraic zero
count defining the twisted Euler number is additive:
\[
e(F)=e(F_1)+e(F_2).
\]

Iterating this construction along the chosen \(r-1\) disjoint arcs yields
\(F_1\#\cdots\#F_r\). At each step the mod-\(2\) homology class and twisted normal
Euler number add, while the Euler characteristic decreases by \(2\); the genus
formula then follows from \(\chi(G)=2-g(G)\) for any connected nonorientable
surface \(G\). This proves the lemma.
\end{proof}

\subsection{Branched covers and signatures}

The next proposition is standard. The existence statement goes back to Fox
\cite{Fox}, while the signature formula is the \(2\)-fold case of the general
branched-cover signature theorem; in the topological locally flat category one
can cite Geske, Kjuchukova, and Shaneson \cite{GeskeKjuchukovaShaneson}.

Here and below, if \(F\subset M\) is any connected locally flat topologically
embedded closed surface, orientable or not, we write \(e(F)\) for the normal Euler
number of \(\nu_M(F)\); if \(F\) is orientable this is the self-intersection
number, and if \(F\) is nonorientable it is the twisted normal Euler number.

For readers who prefer a geometric picture, the branched cover is built in two
stages. First one removes a tubular neighborhood of the branching surface and
constructs the associated double cover of the exterior. Then one fills back in
across the removed neighborhood by the local model \((x,z)\mapsto (x,z^2)\). The
proof below expresses this construction in cohomological terms.

\begin{proposition}\label{prop:branched-cover}
Let \(M\) be a closed connected oriented topological \(4\)-manifold, and let
\(F\subset M\) be a connected locally flat topologically embedded closed surface
with
\[
[F]=0 \in H_2(M;\F_2).
\]
Then there exists a connected \(2\)-fold branched cover
\[
p\colon N\to M
\]
branched along \(F\). If \(A:=p^{-1}(F)\) denotes the ramification surface
upstairs, then \(p|_A\colon A\to F\) is a homeomorphism, \(N\) is a closed
connected oriented topological \(4\)-manifold, and
\begin{equation}\label{eq:sig-general}
\sigma(N)=2\sigma(M)-e(A)=2\sigma(M)-\tfrac12 e(F).
\end{equation}
Consequently,
\begin{equation}\label{eq:sig-abs}
\abs{\sigma(N)-2\sigma(M)}=\tfrac12\abs{e(F)}.
\end{equation}
Moreover,
\begin{equation}\label{eq:chi-branched}
\chi(N)=2\chi(M)-\chi(F).
\end{equation}
\end{proposition}

\begin{proof}
Let \(U\) be a closed tubular neighborhood of \(F\), and set
\[
E:=M\setminus \operatorname{int}(U).
\]
By the long exact sequence of the pair \((M,E)\), together with excision and the
Thom isomorphism (see, for example, \cite[Sections~2.1 and~4.D]{Hatcher})
\[
H^2(M,E;\F_2)\cong H^2(U,\partial U;\F_2)\cong H^0(F;\F_2),
\]
there is an exact segment
\[
H^1(M;\F_2)\longrightarrow H^1(E;\F_2)\longrightarrow H^0(F;\F_2)
\xrightarrow{\delta} H^2(M;\F_2),
\]
where \(\delta(1)=\operatorname{PD}[F]\). Since \([F]=0\) in \(H_2(M;\F_2)\), we
have \(\delta=0\). Because \(F\) is connected, \(H^0(F;\F_2)\cong \F_2\), so its
generator lifts to a class
\[
\phi\in H^1(E;\F_2).
\]
The exterior \(E\) is connected. Indeed, \(M\setminus F\) is connected: any two
points of \(M\setminus F\) can be joined by a path in \(M\), and by general
position the path may be perturbed rel endpoints to avoid the locally flat
codimension-two submanifold \(F\). Moreover, \(M\setminus F\) deformation retracts
onto \(E=M\setminus \operatorname{int}(U)\). Hence \(E\) is connected.

Under the Thom identification, the map \(H^1(E;\F_2)\to H^0(F;\F_2)\) records the
value of a class on a meridian circle of \(F\). Thus \(\phi\) evaluates
nontrivially on every meridian. In particular, \(\phi\neq 0\). Since \(E\) is
connected, the associated homomorphism
\[
\pi_1(E)\to \F_2
\]
is surjective. Hence \(\phi\) determines a connected double cover
\[
\tld{E}\to E;
\]
see, for example, \cite[Section~1.3]{Hatcher}.

By Fox's completion construction \cite{Fox}, this cover extends uniquely over
\(U\) to a branched cover
\[
p\colon N\to M
\]
with local model \((x,z)\mapsto (x,z^2)\); equivalently, away from \(F\) the cover
has two sheets, and along \(F\) those two sheets come together into one. In
particular, \(A:=p^{-1}(F)\) is a connected locally flat surface and
\(p|_A\colon A\to F\) is a homeomorphism. The orientation of \(M\setminus F\)
lifts to an orientation of \(N\setminus A\), and this extends uniquely across the
codimension-two subset \(A\). Thus \(N\) is a closed connected oriented
topological \(4\)-manifold.

The signature formula
\[
\sigma(N)=2\sigma(M)-e(A)
\]
is the \(2\)-fold case of \cite[Theorem~1]{GeskeKjuchukovaShaneson}, since the
only nontrivial branching index is \(r=2\) and
\[
\frac{r^2-1}{3}=\frac{4-1}{3}=1.
\]
Let
\[
q\colon \nu_N(A)\to \nu_M(F)
\]
be the induced map of normal bundles. Via the identification \(A\cong F\), the
map \(q\) covers the identity on \(F\), and in the local model it is fiberwise
\(z\mapsto z^2\), hence has degree \(2\) on each fiber. Therefore the Thom classes
(with the appropriate local coefficient systems) satisfy
\[
q^*u_F=2u_A,
\]
and pulling back by the zero section yields
\[
e(F)=2e(A).
\]
Substituting this into the signature formula gives \eqref{eq:sig-general}, and
\eqref{eq:sig-abs} follows immediately.

For the Euler characteristic, note that
\[
M=E\cup U,
\qquad
N=\tld{E}\cup p^{-1}(U),
\]
where \(U\) and \(p^{-1}(U)\) deformation retract onto \(F\) and \(A\),
respectively. Moreover, \(\partial E\) and \(\partial\tld{E}\) are circle bundles
over closed surfaces, so both have Euler characteristic \(0\). Hence
\[
\chi(M)=\chi(E)+\chi(F)
\]
and
\[
\chi(N)=\chi(\tld{E})+\chi(A)=2\chi(E)+\chi(F),
\]
since \(\chi(A)=\chi(F)\). Therefore
\[
\chi(N)=2(\chi(M)-\chi(F))+\chi(F)=2\chi(M)-\chi(F),
\]
which is \eqref{eq:chi-branched}. This is the usual bookkeeping for a two-fold
branched cover: away from the branch locus the cover is \(2\)-to-\(1\), while
along \(F\) the two sheets coalesce, so one subtracts exactly one copy of
\(\chi(F)\).
\end{proof}

\begin{proposition}[Lee--Weintraub exact sequence]\label{prop:lee-weintraub-enr}
Let \(p\colon N\to M\) be a connected \(2\)-fold branched cover of compact ENR
pairs, with branch locus \(F\subset M\). Then there is a long exact sequence with
\(\F_2\)-coefficients
\[
\cdots \to H_i(M,F;\F_2)\to H_i(N;\F_2)\to H_i(M;\F_2)\to
H_{i-1}(M,F;\F_2)\to \cdots .
\]
\end{proposition}

\begin{proof}
This is the form of \cite[Theorem~1]{LeeWeintraub} used below. Lee and Weintraub
state the result in the simplicial setting, but their proof is chain-level and
uses only the double branched covering structure and transfer over \(\F_2\). The
same argument applies to singular chains for compact ENR pairs.
\end{proof}

We have the following Betti-number estimate.

\begin{lemma}\label{lem:b1-bound}
Let \(p\colon N\to M\) be a connected \(2\)-fold branched cover of a closed
connected topological \(4\)-manifold \(M\), branched along a connected closed
locally flat surface \(F\subset M\). Then
\[
b_1(N;\F_2)\le 2b_1(M;\F_2).
\]
In particular,
\[
b_1(N;\R)\le 2b_1(M;\F_2).
\]
\end{lemma}

\begin{remark}
At this point the mechanism of the proof is visible. After tubing, the same-sign
assumption turns \(\sum_i \abs{e_i}\) into \(\abs{e(F)}\). Proposition~\ref{prop:branched-cover}
converts \(\abs{e(F)}\) into a signature defect of the branched cover, while
\eqref{eq:chi-branched} and Lemma~\ref{lem:b1-bound} bound \(b_2(N)\) by a
quantity linear in \(g(F)\). The theorem follows by comparing these two pieces of
information.
\end{remark}

\begin{proof}
By Proposition~\ref{prop:lee-weintraub-enr}, there is a long exact sequence with
\(\F_2\)-coefficients
\[
\cdots \to H_i(M,F;\F_2)\to H_i(N;\F_2)\to H_i(M;\F_2)\to
H_{i-1}(M,F;\F_2)\to \cdots .
\]
Taking \(i=1\), we obtain
\[
H_1(M,F;\F_2)\to H_1(N;\F_2)\to H_1(M;\F_2)\to H_0(M,F;\F_2).
\]
Since \(F\neq\varnothing\) and \(M\) is connected, \(H_0(M,F;\F_2)=0\). Hence
\[
H_1(M,F;\F_2)\to H_1(N;\F_2)\to H_1(M;\F_2)\to 0
\]
is exact, and therefore
\[
b_1(N;\F_2)\le b_1(M;\F_2)+b_1(M,F;\F_2).
\]

Now consider the long exact sequence of the pair \((M,F)\):
\[
H_1(F;\F_2)\to H_1(M;\F_2)\to H_1(M,F;\F_2)\to
H_0(F;\F_2)\to H_0(M;\F_2).
\]
Because \(F\) and \(M\) are connected, the map
\[
H_0(F;\F_2)\to H_0(M;\F_2)
\]
is an isomorphism. It follows that
\[
H_1(M;\F_2)\to H_1(M,F;\F_2)
\]
is surjective, so
\[
b_1(M,F;\F_2)\le b_1(M;\F_2).
\]
Combining the two inequalities gives
\[
b_1(N;\F_2)\le 2b_1(M;\F_2).
\]

Finally, by the universal coefficient theorem for homology,
\[
H_1(N;\F_2)\cong H_1(N;\Z)\otimes \F_2.
\]
If
\[
H_1(N;\Z)\cong \Z^r\oplus T
\]
with \(T\) finite, then
\[
b_1(N;\F_2)=r+\dim_{\F_2}(T\otimes \F_2)\ge r=b_1(N;\R).
\]
Hence
\[
b_1(N;\R)\le b_1(N;\F_2)\le 2b_1(M;\F_2),
\]
as claimed.
\end{proof}

\section{Proof of the main theorem}

\begin{proof}[Proof of Theorem~\ref{thm:main}]
Apply Lemma~\ref{lem:tubing} to form the ambient connected sum
\[
F:=F_1\#\cdots\#F_r\subset M.
\]
Then \(F\) is connected and nonorientable, and
\[
g(F)=\sum_{i=1}^r g_i,\qquad
[F]=[F_1]+\cdots+[F_r]=0\in H_2(M;\F_2),\qquad
e(F)=\sum_{i=1}^r e_i.
\]
Because the integers \(e_1,\dots,e_r\) all have the same sign, we have
\[
\abs{e(F)}=\sum_{i=1}^r \abs{e_i}.
\]

By Proposition~\ref{prop:branched-cover}, there exists a connected \(2\)-fold
branched cover
\[
p\colon N\to M
\]
branched along \(F\), and
\[
\frac12\sum_{i=1}^r \abs{e_i}
=
\frac12\abs{e(F)}
=
\abs{\sigma(N)-2\sigma(M)}.
\]
Hence
\begin{equation}\label{eq:sigestimate-proof}
\frac12\sum_{i=1}^r \abs{e_i}
\le
\abs{\sigma(N)}+2\abs{\sigma(M)}.
\end{equation}
Since
\[
\abs{\sigma(N)}\le b_2(N;\R)\le b_2(N;\F_2),
\]
it remains to bound \(b_2(N;\F_2)\).

Again by Proposition~\ref{prop:branched-cover},
\[
\chi(N)=2\chi(M)-\chi(F)=2\chi(M)+g(F)-2.
\]
Because \(N\) is a closed connected \(4\)-manifold, Poincar\'e duality over \(\F_2\)
(see, for example, \cite[Section~3.3]{Hatcher}) gives
\[
\chi(N)=2-2b_1(N;\F_2)+b_2(N;\F_2).
\]
Therefore
\[
b_2(N;\F_2)=\chi(N)-2+2b_1(N;\F_2).
\]
Using Lemma~\ref{lem:b1-bound}, we obtain
\begin{align*}
b_2(N;\F_2)
&\le \bigl(2\chi(M)+g(F)-2\bigr)-2+4b_1(M;\F_2)\\
&=2\chi(M)+g(F)-4+4b_1(M;\F_2).
\end{align*}
Substituting this into \eqref{eq:sigestimate-proof} yields
\[
\frac12\sum_{i=1}^r \abs{e_i}
\le
2\chi(M)+g(F)-4+4b_1(M;\F_2)+2\abs{\sigma(M)}.
\]
Since \(M\) is a closed connected \(4\)-manifold,
\[
b_2(M;\F_2)=\chi(M)-2+2b_1(M;\F_2),
\]
so the previous inequality becomes
\[
\frac12\sum_{i=1}^r \abs{e_i}
\le
g(F)+2b_2(M;\F_2)+2\abs{\sigma(M)}.
\]
Recalling that \(g(F)=\sum_{i=1}^r g_i\) and multiplying by \(2\), we get
\[
\sum_{i=1}^r \abs{e_i}
\le
2\sum_{i=1}^r g_i + 4\abs{\sigma(M)}+4b_2(M;\F_2).
\]
Finally, using
\[
4b_2(M;\F_2)=8b_1(M;\F_2)+4\chi(M)-8,
\]
this is exactly \eqref{eq:mainbound}. Subtracting \(2\sum_{i=1}^r g_i\) from both
sides gives \eqref{eq:excessbound}.
\end{proof}

\section{Projective planes and plane bundles}

We now deduce the projective-plane version.

\begin{proof}[Proof of Corollary~\ref{cor:projective-planes}]
Let
\[
P_1,\dots,P_m\subset M
\]
be pairwise disjoint locally flat topologically embedded copies of \(\RP^2\) with
\(\abs{e(P_i)}>2\). At least \(\lceil m/2\rceil\) of the integers \(e(P_i)\) have
the same sign; after reindexing, assume this is true for
\[
P_1,\dots,P_s,
\qquad s\ge \left\lceil \frac m2\right\rceil.
\]

Set
\[
k:=b_2(M;\F_2)
\]
and
\[
D(M):=4\abs{\sigma(M)}+8b_1(M;\F_2)+4\chi(M)-8.
\]
Since \(M\) is a closed connected \(4\)-manifold,
\[
D(M)=4\abs{\sigma(M)}+4b_2(M;\F_2)\ge 0.
\]

If \(s\le k\), then
\[
m\le 2s\le 2k.
\]
So it remains to consider the case \(s>k\). By Lemma~\ref{lem:zero-sum}, among the
mod-\(2\) classes
\[
[P_1],\dots,[P_s]\in H_2(M;\F_2)
\]
there is a nonempty zero-sum subcollection of size at least \(s-k\). After
relabeling, assume
\[
[P_1]+\cdots+[P_n]=0,
\qquad
n\ge s-k>0.
\]

Now apply Theorem~\ref{thm:main} to \(P_1,\dots,P_n\). Since each \(P_i\cong
\RP^2\) has nonorientable genus \(1\) and \(\abs{e(P_i)}>2\), each term
\[
\abs{e(P_i)}-2
\]
is at least \(1\). Hence
\[
n\le \sum_{i=1}^n \bigl(\abs{e(P_i)}-2\bigr)\le D(M).
\]
Therefore
\[
\left\lceil \frac m2\right\rceil-k\le s-k\le n\le D(M),
\]
so
\[
m\le 2\bigl(k+D(M)\bigr).
\]

Thus the corollary holds, for example with
\[
B(M):=2\bigl(b_2(M;\F_2)+D(M)\bigr).
\]
\end{proof}

\begin{proof}[Proof of Corollary~\ref{cor:plane-bundles}]
For the first assertion, let
\[
U_1,\dots,U_m\subset M
\]
be pairwise disjoint tubular neighborhoods of locally flat topologically embedded
copies
\[
P_1,\dots,P_m\subset M
\]
of \(\RP^2\), and assume that \(\abs{e(P_i)}>2\) for each \(i\). Since the
\(U_i\) are pairwise disjoint, the zero-sections \(P_i\subset U_i\) are pairwise
disjoint locally flat topologically embedded copies of \(\RP^2\) in \(M\). By
Corollary~\ref{cor:projective-planes}, there can be only finitely many such
\(P_i\), and hence only finitely many such pairwise disjoint tubular neighborhoods.

For the second assertion, let
\[
V_1,\dots,V_m\subset M
\]
be pairwise disjoint open subsets, and suppose that each \(V_i\) is homeomorphic to
the total space of a real \(2\)-plane bundle
\[
\pi_i\colon E_i\to \RP^2
\]
whose total space is orientable and whose twisted Euler number has absolute value
greater than \(2\). Choose a homeomorphism
\[
\phi_i\colon E_i\to V_i.
\]
Let \(Z_i\subset E_i\) denote the zero-section and set
\[
P_i:=\phi_i(Z_i)\subset V_i\subset M.
\]
Then the \(P_i\) are pairwise disjoint locally flat topologically embedded copies
of \(\RP^2\).

Moreover, since \(V_i\) is open in \(M\), we have
\[
\nu_M(P_i)\cong \nu_{V_i}(P_i).
\]
Under the homeomorphism \(\phi_i\), the normal bundle \(\nu_{V_i}(P_i)\) is
identified with the normal bundle of the zero-section \(Z_i\subset E_i\), and the
latter is canonically isomorphic to \(E_i\) itself. Therefore the twisted normal
Euler number of \(P_i\) in \(M\) is exactly \(e_{\mathrm{tw}}(E_i)\), so in
particular
\[
\abs{e(P_i)}>2.
\]
Applying Corollary~\ref{cor:projective-planes} again, we conclude that there can
be only finitely many such pairwise disjoint open subsets \(V_i\).

Equivalently, after choosing bundle metrics on the \(E_i\), the images under the
\(\phi_i\) of the closed unit disk bundles are pairwise disjoint compact tubular
neighborhoods in \(M\), reducing the second assertion to the first.
\end{proof}

\section{A connected-boundary packing counterexample}

Before giving the connected-boundary counterexample, we first recall that general
codimension-zero packing phenomena can be very different without orientability or
without a positive normal-Euler-excess hypothesis. For example, the nonorientable
\(4\)-manifold \(\RP^2\times S^2\) contains infinitely many pairwise disjoint copies
of \(\RP^2\times D^2\): if \(D_1,D_2,\dots\subset S^2\) are pairwise disjoint
embedded closed disks, then
\[
\RP^2\times D_j\subset \RP^2\times S^2
\]
are pairwise disjoint and each is homeomorphic to \(\RP^2\times D^2\). This example
is not a counterexample to Corollary~\ref{cor:plane-bundles}; rather, it
illustrates why the codimension-zero packing question considered below is broader
than the positive-excess disk-bundle situation.

Even in the oriented category, however, connected boundary alone is not enough to
force an asymptotically packable compact \(4\)-manifold to embed in \(\R^4\). The
following question is a natural strengthening of the disconnected-boundary product
examples.

\begin{question}\label{ques:connected-boundary-packing}
Does there exist a compact connected oriented \(4\)-manifold \(W\) with connected
boundary and a compact \(4\)-manifold \(N\) such that, for every \(n\), there are
pairwise disjoint codimension-zero embeddings
\[
\bigsqcup_{i=1}^n W\hookrightarrow \operatorname{int}(N),
\]
but \(W\) does not embed in \(\R^4\)?
\end{question}

The answer is yes.

Let
\[
P:=\RP^3\setminus \operatorname{int}B^3,
\]
where \(B^3\subset \RP^3\) is a smoothly embedded closed \(3\)-ball, and let
\[
W:=P\,\widetilde{\times}\,[0,1]
\]
denote the compact \(4\)-manifold obtained from the product with corners
\(P\times[0,1]\) by rounding the corners. Thus \(W\) is diffeomorphic to any
standard smoothing of \(P\times[0,1]\) near \(\partial P\times\{0,1\}\).

We shall use the following elementary collar-shrinking observation. If a compact
topological \(4\)-manifold \(Q\) with boundary embeds topologically in the
interior of a topological \(4\)-manifold \(X\), then a collar of \(\partial Q\) in
\(Q\) may be pushed slightly inward. The resulting collar-shrunken copy is
homeomorphic to \(Q\), has bicollared boundary in \(X\), and is contained in the
original image. Thus, for codimension-zero embeddings of compact \(4\)-manifolds,
a separate locally flat hypothesis is redundant after replacing the image by a
homeomorphic collar-shrunken copy.

Recall that if \(Y\) is a closed oriented \(3\)-manifold, its torsion linking form
is the nonsingular symmetric pairing
\[
\lambda_Y\colon \operatorname{Tor}H_1(Y;\Z)\times
\operatorname{Tor}H_1(Y;\Z)\to \Q/\Z.
\]
We call \(\lambda_Y\) \emph{hyperbolic} if
\[
\operatorname{Tor}H_1(Y;\Z)=A\oplus B
\]
for subgroups \(A\) and \(B\) on which \(\lambda_Y\) vanishes identically.

\begin{proposition}[Hantzsche obstruction]\label{prop:hantzsche}
If a closed connected oriented \(3\)-manifold \(Y\) embeds locally flatly in
\(S^4\), then its torsion linking form \(\lambda_Y\) is hyperbolic.
\end{proposition}

\begin{proof}
Let \(S^4=A\cup_Y B\), where \(A\) and \(B\) are the closures of the two
complementary components of \(S^4\setminus Y\). Since \(Y\subset S^4\) is locally
flat, \(A\) and \(B\) are compact oriented topological \(4\)-manifolds with common
boundary \(Y\).

The Mayer--Vietoris sequence for \(S^4=A\cup_Y B\), using
\[
H_1(S^4;\Z)=H_2(S^4;\Z)=0,
\]
gives an isomorphism
\[
H_1(Y;\Z)\cong H_1(A;\Z)\oplus H_1(B;\Z).
\]
Restricting to torsion gives
\[
\operatorname{Tor}H_1(Y;\Z)\cong
\operatorname{Tor}H_1(A;\Z)\oplus \operatorname{Tor}H_1(B;\Z).
\]
Let
\[
L_A:=\ker\{\operatorname{Tor}H_1(Y;\Z)\to \operatorname{Tor}H_1(A;\Z)\}
\]
and
\[
L_B:=\ker\{\operatorname{Tor}H_1(Y;\Z)\to \operatorname{Tor}H_1(B;\Z)\}.
\]
Under the above splitting, \(L_A\) and \(L_B\) are complementary subgroups. It
remains to show that both are isotropic for the linking form.

We prove this for \(L_A\); the proof for \(L_B\) is identical. The same
Mayer--Vietoris sequence shows that
\[
H_2(Y;\Z)\to H_2(A;\Z)
\]
is surjective. Hence the long exact sequence of the pair \((A,Y)\) implies that
\[
H_2(A,Y;\Z)\to H_1(Y;\Z)
\]
has image equal to \(\ker\{H_1(Y;\Z)\to H_1(A;\Z)\}\). Thus \(L_A\) is the
torsion part of the image of \(H_2(A,Y;\Z)\to H_1(Y;\Z)\).

By Poincar\'e--Lefschetz duality for \(A\) and Poincar\'e duality for \(Y\), this
torsion image identifies with the image of the restriction map
\[
i^*\colon \operatorname{Tor}H^2(A;\Z)\to \operatorname{Tor}H^2(Y;\Z),
\]
where \(i\colon Y\hookrightarrow A\) is inclusion.

We use the standard cohomological formula for the torsion linking form. If
\(x,y\in \operatorname{Tor}H^2(Y;\Z)\), then
\[
\lambda_Y(x,y)=\langle x\smile \widetilde y,[Y]\rangle,
\]
where \(\widetilde y\in H^1(Y;\Q/\Z)\) maps to \(y\) under the Bockstein for
\[
0\to \Z\to \Q\to \Q/\Z\to 0.
\]
Take \(x=i^*\alpha\) and \(y=i^*\beta\), where
\[
\alpha,\beta\in \operatorname{Tor}H^2(A;\Z).
\]
Choose \(\widetilde\beta\in H^1(A;\Q/\Z)\) whose Bockstein is \(\beta\). By
naturality, \(i^*\widetilde\beta\) is a Bockstein lift of \(y=i^*\beta\). Therefore
\[
\lambda_Y(x,y)
=
\langle i^*\alpha\smile i^*\widetilde\beta,[Y]\rangle
=
\langle \alpha\smile\widetilde\beta,i_*[Y]\rangle
=0,
\]
because \(i_*[Y]=0\in H_3(A;\Z)\). Hence \(L_A\) is isotropic. Similarly \(L_B\)
is isotropic. Therefore the torsion linking form of \(Y\) is hyperbolic.
\end{proof}

\begin{lemma}\label{lem:rp3-sum-not-hyperbolic}
The torsion linking form of \(\RP^3\#(-\RP^3)\) is
\[
\left(\frac12\right)\oplus\left(\frac12\right)
\]
on \((\Z/2)^2\), up to the usual overall sign convention. This form is not
hyperbolic.
\end{lemma}

\begin{proof}
The standard surgery description of \(\RP^3=L(2,1)\) gives
\[
H_1(\RP^3;\Z)\cong \Z/2
\]
with generator self-linking \(1/2\in \Q/\Z\). Reversing orientation changes the
linking form by a sign, but \(-1/2=1/2\) in \(\Q/\Z\). Since the linking form of a
connected sum is the orthogonal direct sum of the linking forms of the summands,
the linking form of \(\RP^3\#(-\RP^3)\) is
\[
\left(\frac12\right)\oplus\left(\frac12\right)
\]
on \((\Z/2)^2\).

Let \(e_1,e_2\) be the standard basis of \((\Z/2)^2\). Then
\[
\lambda(e_1,e_1)=\frac12,
\qquad
\lambda(e_2,e_2)=\frac12,
\qquad
\lambda(e_1+e_2,e_1+e_2)=\frac12+\frac12=0\in\Q/\Z.
\]
Thus the only nonzero isotropic subgroup is
\[
\langle e_1+e_2\rangle.
\]
A hyperbolic splitting of a group of order \(4\) would require two complementary
isotropic subgroups of order \(2\). The only order-\(2\) subgroups complementary to
\(\langle e_1+e_2\rangle\) are \(\langle e_1\rangle\) and \(\langle e_2\rangle\),
and neither is isotropic. Hence the form is not hyperbolic.
\end{proof}

\begin{theorem}\label{thm:connected-boundary-counterexample}
The manifold \(W\) is compact, connected, and oriented, and \(\partial W\) is
connected. Moreover, if \(N\) is a compact oriented topological \(4\)-manifold
whose interior contains a bicollared embedded copy of \(\RP^3\), then for every
\(n\ge 1\) there are pairwise disjoint codimension-zero embeddings
\[
\bigsqcup_{i=1}^n W\hookrightarrow \operatorname{int}(N).
\]
In particular, this holds for \(N=\RP^3\times S^1\). Nevertheless \(W\) does not
embed in \(\R^4\). If \(N\) is smooth and the embedded copy of
\(\RP^3\subset \operatorname{int}(N)\) is smooth, then the packings above may be
chosen smooth.
\end{theorem}

\begin{proof}
The product \(P\times[0,1]\) is compact, connected, and oriented, and rounding
corners does not change these properties. Its boundary is the union
\[
(P\times\{0\})\cup(\partial P\times[0,1])\cup(P\times\{1\}),
\]
with corners rounded. Since \(P\) is connected and \(\partial P\cong S^2\) is
nonempty, this boundary is connected. Equivalently,
\[
\partial W\cong D(P)\cong \RP^3\#(-\RP^3),
\]
the double of \(P\).

Now suppose \(\RP^3\subset \operatorname{int}(N)\) is bicollared. Since both
\(\RP^3\) and \(N\) are oriented, this hypersurface is two-sided, so after choosing
the side determined by the orientations it has a product bicollar
\[
\RP^3\times(-\varepsilon,\varepsilon)\subset \operatorname{int}(N).
\]
Choose pairwise disjoint compact intervals
\[
J_1,\dots,J_n\subset(-\varepsilon,\varepsilon).
\]
For each \(j\), the product
\[
P\times J_j\subset \RP^3\times J_j
\]
with its corners rounded is a codimension-zero submanifold of \(N\), homeomorphic
to \(W\). These rounded products are pairwise disjoint, giving the required
packing. If \(N\) is smooth and the embedded copy of \(\RP^3\) is smooth, the
bicollar may be chosen smooth, and the same construction gives smooth
codimension-zero embeddings.

It remains to show that \(W\) does not embed in \(\R^4\). If such an embedding
existed, the collar-shrinking observation above would give a homeomorphic copy of
\(W\) embedded in \(\R^4\) with locally flat boundary. Hence \(\partial W\) would
embed locally flatly in \(S^4\). By the boundary calculation above,
\[
\partial W\cong \RP^3\#(-\RP^3).
\]
By Proposition~\ref{prop:hantzsche}, the torsion linking form of any closed
oriented \(3\)-manifold embedded locally flatly in \(S^4\) is hyperbolic. But
Lemma~\ref{lem:rp3-sum-not-hyperbolic} shows that the torsion linking form of
\(\RP^3\#(-\RP^3)\) is not hyperbolic. This contradiction proves that no embedding
\(W\hookrightarrow \R^4\) exists.
\end{proof}

\begin{remark}\label{rem:whitney-obstruction}
There is also a very short smooth obstruction, which explains geometrically why
this example is incompatible with \(\R^4\). Choose the standard embedded
\(\RP^2\subset \RP^3\), and choose the ball \(B^3\) disjoint from it. Then
\(\RP^2\subset P\), and hence
\[
\RP^2\times[0,1]\subset P\times[0,1].
\]
The middle slice \(\RP^2\times\{1/2\}\subset W\) has a normal bundle with a global
nowhere-zero section, namely the section in the interval direction. Thus its
twisted normal Euler number in \(W\) is \(0\). If \(W\) embedded smoothly in
\(\R^4\), this would produce a smooth embedding \(\RP^2\subset \R^4\) with twisted
normal Euler number \(0\), contradicting Whitney's congruence
\[
e(\RP^2\subset\R^4)\equiv 2\pmod 4
\]
for smooth embeddings \cite{Whitney1941,Massey}.
\end{remark}

\begin{remark}\label{rem:counterexample-context}
Theorem~\ref{thm:connected-boundary-counterexample} answers
Question~\ref{ques:connected-boundary-packing} in the affirmative and supersedes
the simpler disconnected-boundary lens-space product example. It is also
compatible with the earlier obstructions in \cite{CFSZCorr}: for
\(\partial W\cong \RP^3\#(-\RP^3)\), the linking form
\((1/2)\oplus(1/2)\) is split metabolic, with diagonal metabolizer, but it is not
hyperbolic. Thus the example satisfies the split-metabolic conclusion forced by
asymptotic packing while failing the stronger boundary condition forced by an
actual embedding in \(S^4\). Finally, the normal-Euler excess theorem above does
not apply to this example. The obstruction to \(W\subset \R^4\) is the boundary
linking form of \(\partial W\). In the smooth category,
Remark~\ref{rem:whitney-obstruction} gives an additional way to see the same
nonembedding phenomenon, namely the presence of an \(\RP^2\) whose twisted normal
Euler number would be \(0\) after embedding in \(\R^4\).
\end{remark}

\begin{acknowledgement}
The authors used ChatGPT as a tool for error detection, proof checking, and for
helping sharpen the paper from the projective-plane case to the more natural
general formulation in terms of disjoint nonorientable surfaces and identify the connected-boundary
counterexample.
\end{acknowledgement}

\end{document}